\documentclass[12pt]{amsart}
\usepackage{amssymb}
\allowdisplaybreaks

\usepackage{color}
\usepackage{graphics}
\usepackage{amsmath,amscd}
\usepackage{amsthm}
\usepackage[all,cmtip]{xy}
\usepackage{tikz}
\usepackage{graphicx}
\usepackage{epstopdf}
\usepackage{mathtools}
\usepackage{amssymb}
\usepackage{amsbsy}
\usepackage{bm}
\usepackage{tikz-cd}
\usepackage{enumerate}
\usepackage{enumitem}
\usepackage{float}
\usepackage[inkscapearea=page]{svg}
\theoremstyle{definition}

\usepackage[colorlinks=green,linkcolor=blue,citecolor=red,urlcolor=red]{hyperref}

\newcommand{\R}{\mathbb{R}}

\newcommand{\D}{\mathbb{D}}

\newcommand{\s}{\mathbb{S}}

\newcommand{\norm}[1]{\left\lVert#1\right\rVert}

\newtheorem{proposition}{Proposition}[section]
\newtheorem{theorem}[proposition]{Theorem}

\newtheorem{lemma}[proposition]{Lemma}
\newtheorem{conjecture}[proposition]{Conjecture}
\newtheorem{corollary}[proposition]{Corollary}

\newtheorem{question}[proposition]{Question}

\begin{document}
	
	\title{$4$-manifolds and non-existence of open book}
	
	
	
	\author{Shital Lawande}
	\address{IAI TCG CREST Kolkata, and Ramakrishna Mission Vivekananda Education and Research Institute, Belur Math}
	\email{shital.lawande@tcgcrest.org}
	
	\author{Kuldeep Saha}
	\address{IAI TCG CREST Kolkata, and Academy of Scientific and Innovative Research, Gazhiabad}
	\email{kuldeep.saha@gmail.com, kuldeep.saha@tcgcrest.org}

	\begin{abstract}
		
	Following a recent work of Kastenholz \cite{Kast}, we show existence of infinitely many parallelizable closed oriented $4$-manifolds, none of which admit an open book decomposition. This implies there is no analogue of the Giroux correspondence for Engel manifolds. In particular, not every Engel manifold has a supporting open book, in the sense of Colin, Presas and Vogel \cite{cpv}.   
	
	\end{abstract}
	
	\maketitle

	\section{Introduction}

	\noindent An open book decomposition of a manifold $M^m$ is a pair $(V^{m-1},h)$, such that $M^m$ is diffeomorphic to the quotient space $\mathcal{MT}(V^{m-1}, h) \cup_{id} \partial V^{m-1} \times D^2$. Here, $V^{m-1}$, called \emph{page}, is a manifold with boundary. The map $h$, called \emph{monodromy}, is a diffeomorphism of $V^{m-1}$ that restricts to identity near the boundary $\partial V$, and $\mathcal{MT}(V^{m-1}, h)$ denotes the mapping torus of $h$. It is well known that every closed, orientable, odd dimensional manifold admits open book decomposition. For closed orientable manifolds of even dimension greater than $4$, Quinn's signature invariant is the only obstruction to existence of an open book decomposition \cite{Qu}. For a long time it was not known whether the vanishing of signature is a sufficient condition in dimension $4$. A recent work of Kastenholz \cite{Kast}, has proved that vanishing of signature is not enough for existence of an open book decomposition. 
	
	\begin{theorem}[Kastenholz \cite{Kast}] \label{thmkast}
		
		A $4$-manifold admitting open book decomposition must have zero simplicial volume.
		
	\end{theorem}
	
	\noindent The product of two surfaces of genus greater than equal to two has signature zero, but has nonzero simplicial volume. The notion of \emph{simplicial volume} was introduced by Gromov \cite{Gr}. For a topological space $X$, let $c = \sum_{i} r_i \sigma_i$ ($\sigma_i$s singular simplices) be a real $k$-chain in $C_k(X;\R)$. The \emph{simplicial $l^1$-norm} of $c$ is defined to be $\norm{c} = \sum_{i}|r_i|$.  The \emph{Gromov norm} of a homology class $\alpha \in H_k(X;\R)$ is then defined as $\norm{\alpha} = \inf_z \norm{z}$, where $z$ runs over all singular cycles representing $\alpha$. For a closed orientable manifold $V$ its \emph{simplicial volume} $\norm{V}$ is defined as the Gromov norm
	of its fundamental class. It is known that a hyperbolic manifold has non-zero simplicial volume (see sections $0.2$ and $0.3$ in \cite{Gr}). Since, $\Sigma_g \times \Sigma_h$ is hyperbolic for $h\geq2$ and $g \geq 2$, by Theorem \ref{thmkast} it can not admit an open book decomposition, despite having signature zero.

	\noindent As an application of Theorem \ref{thmkast}, we prove the following.

	\begin{theorem} \label{mainthm}
		
	There exist infinitely many parallelizable closed orinted $4$-manifolds, none of which admit open book decomposition. 
		
	\end{theorem} 
	
	 \noindent The Thurston--Winkelenkemper \cite{TW} construction associates a contact structure to an open book with symplectic page and a symplectomorphism as monodromy. The Giroux correspondence \cite{Gi} says that the converse is also true. In particular, every contact structure is supported by an open book. Colin, Presas and Vogel \cite{cpv} investigated an analogue of the Giroux correspondence for Engel manifolds. Vogel \cite{Vog} has shown that every parallelizable $4$-manifold admits an Engel structure and vice versa. Colin, Presas and Vogel \cite{cpv} gave an analogue of the Thurston--Winkelenkemper construction for Engel manifolds by associating an Engel structure to an open book with page a contact $3$-manifold with tori boundary components and monodromy which preserves a framing on the page. They asked the following.
	
	\begin{question}(Question $5.6$ in \cite{cpv}) \label{qngc}
		Is every Engel structure homotopic to an Engel structure carried by an
		open book ?
	\end{question}
	
	\noindent Theorem \ref{mainthm} answers Question \ref{qngc} in negative. 
	
	\begin{corollary}\label{gcfail}
		There exist infinitely many Engel manifolds that admit no supporting open book.
	\end{corollary}

	\section{Proof of Theorem \ref{mainthm}}

	\begin{lemma} \label{mainlemma}
		Let $M^4$ be a stably parallelizable closed orientable $4$-manifold such that $\chi(M)$ is even and $\norm{M} > 0$. Then, there exists a closed orientable $4$-manifold $X$ such that $X \# M$ is parallelizable and $\norm{X \# M} > 0$. 
	\end{lemma}
	
	\begin{proof}
		The obstruction to parallelizability of a stably parallelizable closed $4$-manifold is its Euler characteristic. Let $\chi(M) = 2k$. When $k=0$, we take $X$ to be empty set and the result follows.  For $k \geq 1$, we take $X = \#^k \s^1 \times \s^3$. Since $\chi(\s^1 \times \s^3 \# M) = \chi(M) - 2$, we get $\chi(X \# M) = 0$. For $k \leq -1$, we take $X = \#^{|k|} \s^2 \times \s^2$. Since $\chi(\s^2 \times \s^2 \# M) = \chi(M) +2$, we again have $\chi(X \#M) = 0$. Both $\s^1 \times \s^3$ and $\s^2 \times \s^2$ are stably parallelizable, as they smoothly embed in $\R^5$. Since stable parallelizability is preserved under connected sum (see proposition $1.3$ in \cite{jw}), $X \# M$ is also stably parallelizable with $\chi(X\#M) = 0$. Thus, $X \# M$ is parallelizable. By Gromov (section $3.5$ in \cite{Gr}), $\norm{X \# M} = \norm{X} + \norm{M}$. Now, both $\s^1 \times \s^3$ and $\s^2 \times \s^2$ admits open book decompositions with page $\s^1 \times \D^2$ and $\s^2 \times [0,1]$, respectively, with the identity map as monodromies. Taking boundary connected sum of their pages induces open book structures on their connected sums. Therefore, by Theorem \ref{thmkast}, $\norm{X} = 0$. Hence, $\norm{X \# M} > 0$.   
	\end{proof}

	\noindent The proof of Theorem \ref{mainthm} now follows from the fact that $\{\Sigma_g \times \Sigma_h\}_{g \geq 2,h \geq 2}$ gives an infinite family of stably parallelizable closed oriented $4$-manifolds with non-zero simplicial volume and with distinct Euler characteristics of the form $4(g-1)(h-1)$. Among the corresponding parallelizable manifolds $X_{g,h} \# \Sigma_g \times \Sigma_h$, as constructed in Lemma  \ref{mainlemma}, infinitely many can be distinguished by the number of free generators in their fundamental groups.

	\section{Open books and a conjecture of Gromov}
	
	Kastenholz's work connects the existence question for $4$-dimensional open books to the following conjecture of Gromov (conjecture $2$ in \cite{crw}). 
	
	\begin{conjecture}(Gromov) \label{congro}
		If $M$ is aspherical and $ \norm{M} = 0$, then $\chi(M) = 0$.
	\end{conjecture}
	
	\noindent If Conjecture \ref{congro} is true in dimension $4$, then along with Theorem \ref{thmkast}, it implies that an aspherical closed oriented $4$-manifold $V$ can admit open book decomposition only if $\chi(V) = 0$. For example, Ratcliffe and Tscantz \cite{RT} have constructed an inifinite family of aspherical closed $4$-manifolds which are integral homology $4$-spheres. A positive answer to Conjecture \ref{congro} will imply that none of these aspherical manifolds can admit open book decomposition. On the other hand, existence of open book on a $4$-dimensional aspherical manifold with non-zero Euler characteristic will provide a counter example to Conjecture \ref{congro} in dimension $4$.

	\subsection{Acknowledgement} The authors thank Chun-Sheng Hsueh for introducing them to the work of Kastenholz \cite{Kast}.

	\section{Preliminaries}	\label{sec2}

\end{document}